\documentclass[a4paper,12pt]{amsart}
\usepackage{times} 
\usepackage{latexsym,amscd,amssymb,url}
\usepackage[latin1]{inputenc} 	
\usepackage[T1]{fontenc}         
\usepackage{ae}									
\usepackage[all,cmtip]{xy}  
\usepackage{graphicx}
\usepackage{MnSymbol}
\usepackage{hyperref} 
\usepackage{enumerate}

\newtheorem{theorem}{Theorem}

\theoremstyle{plain}

\newtheorem{corollary}[theorem]{Corollary}

\newtheorem{example}[theorem]{Example}

\newtheorem{lemma}[theorem]{Lemma}

\newtheorem{proposition}[theorem]{Proposition}
\newtheorem{remark}[theorem]{Remark}

\newcommand{\Q}{\mathbb Q}

\newcommand{\C}{\mathbb C}

\newcommand{\R}{\operatorname{\textbf{R}}}

\newcommand{\CP}{\mathbb P}

\newcommand{\D}{\operatorname{D}}

\newcommand{\Pic}{\operatorname{Pic}}

\newcommand{\id}{\operatorname{id}}

\newcommand{\pr}{\operatorname{pr}}

\newcommand{\PD}{\operatorname{PD}}

\newcommand{\codim}{\operatorname{codim}}
\newcommand{\red}{\operatorname{red}}
\newcommand{\Gr}{\operatorname{Gr}}
\newcommand{\sm}{\operatorname{sm}}

\newcommand{\dashedlongrightarrow}{\xymatrix@1@=15pt{\ar@{-->}[r]&}}
\renewcommand{\longrightarrow}{\xymatrix@1@=15pt{\ar[r]&}}
\renewcommand{\mapsto}{\xymatrix@1@=15pt{\ar@{|->}[r]&}}
\renewcommand{\twoheadrightarrow}{\xymatrix@1@=15pt{\ar@{->>}[r]&}}
\newcommand{\hooklongrightarrow}{\xymatrix@1@=15pt{\ar@{^(->}[r]&}}
\newcommand{\congpf}{\xymatrix@1@=15pt{\ar[r]^-\sim&}}
\renewcommand{\cong}{\simeq}

\begin{document} 
\title[Decomposable theta divisors and generic vanishing]{Decomposable theta divisors and generic vanishing}

\author{Stefan Schreieder}
\address{Mathematisches Institut, Universit\"at Bonn, Endenicher Allee 60, 53115 Bonn, Germany} 
\email{schreied@math.uni-bonn.de} 

\date{June 20, 2016; \copyright{\ Stefan Schreieder 2016}}
%
%


\begin{abstract}    
We study ample divisors $X$ with only rational singularities on abelian varieties that decompose into a sum of two lower dimensional subvarieties, $X=V+W$.   
For instance, we prove an optimal lower bound on the degree of the addition map $V\times W\longrightarrow X$ and show that the minimum can only be achieved if $X$ is a theta divisor.  
Conjecturally, the latter happens only on Jacobians of curves and intermediate Jacobians of cubic threefolds.
As an application, we prove that nondegenerate generic vanishing subschemes of indecomposable principally polarized abelian varieties are automatically reduced and irreducible, have the expected geometric genus, and property $(\mathcal P)$ with respect to their theta duals. 
\end{abstract}

\maketitle

\section{Introduction}

This paper is motivated by a number of conjectures that relate the existence of special subvarieties of abelian varieties to the Schottky problem. 
Most prominently, Debarre's minimal class conjecture \cite{debarre} states that a $g$-dimensional principally polarized abelian variety (ppav) $(A,\Theta)$ contains a subvariety $V\subset A$ of minimal cohomology class $\frac{\theta^{g-d}}{(g-d)!}$ with $1\leq d\leq g-2$, if and only if  one of the following holds:
\begin{enumerate}[(a)]
\item there is a smooth projective curve $C$ and an isomorphism $(A,\Theta)\cong (JC,\Theta_C)$ which identifies $V$ with the Brill--Noether locus $W_d(C)$; \label{item:JC}
\item $g=5$, $d=2$ and there is a smooth cubic threefold $Y\subset \CP^4$ and an isomorphism $(A,\Theta)\cong (JY,\Theta_Y)$ which identifies $V$ with the Fano surface of lines $F\subset JY$ of $Y$.\label{item:JY}
\end{enumerate}
The famous Matsusaka--Ran criterion establishes the conjecture for $d=1$; further evidence was given by Debarre \cite{debarre}, who proved the conjecture for Jacobians of curves.  
There is also a surprising analogy with projective space \cite[\S 2(6)]{pareschi-popa3}, where it is known \cite{eisenbud-harris} that a nondegenerate subvariety $V\subset \CP^n$ of dimension $1\leq d\leq n-2$ has minimal degree if and only if it is (a cone over) a rational normal scroll or the Veronese image of $\CP^2$ in $\CP^5$.

A common feature of the minimal class subvarieties $V\subset (A,\Theta)$ in (\ref{item:JC}) and (\ref{item:JY}) is that in both cases there is a second minimal class subvariety $W$ such that
\begin{align}\label{eq:V+W}
V+W=\Theta .
\end{align} 
If $V=W_d(C)$, then $W=W_{g-d-1}(C)$, and if $V$ is the Fano surface $F$, then $W=-F$ \cite{clemens-griffiths}. 
Conjecturally, these are the only examples where an irreducible theta divisor decomposes into a sum of two lower dimensional subvarieties, see \cite[Section 8.3]{pareschi-popa}.
If $V$ or $W$ is a curve, this was recently proven in \cite{schreieder}, but it is wide open otherwise. 

In this paper we prove some basic properties of decompositions as in (\ref{eq:V+W}). 
Our method works quite generally for decompositions of arbitrary ample divisors of abelian varieties with at most rational singularities.
This includes theta divisors of indecomposable ppavs by a result of Ein and Lazarsfeld \cite{ein-laz}. 

\begin{theorem} \label{thm:P}  
Let $A$ be a $g$-dimensional abelian variety, and let $V,W\subseteq A$ be closed subvarieties of positive dimensions $d$ and $g-d-1$, respectively. 
If $X:=V+W$ is an ample divisor on $A$ with at most rational singularities, then we have the following:
\begin{enumerate}
	\item the subvarieties $V$ and $W$ are nondegenerate; \label{item:P:1}
	\item the degree of the addition map $f:V\times W\longrightarrow X$ satisfies $\deg(f)\geq \binom{g-1}{d}$; 
	 \label{item:P:2}
	\item if $f$ has minimal degree $\deg(f)=\binom{g-1}{d}$, then
	\begin{enumerate}[(i)]
	\item $X$ is a theta divisor on $A$, i.e.\ $h^0(A,\mathcal O_A(X))=1$; \label{item:P:3:0}
	\item the geometric genera are given by $p_g(V)=\binom{g}{d}$ and $p_g(W)=\binom{g}{d+1}$;  \label{item:P:3}
	\item 
	$V$ has property $(\mathcal P)$ with respect to $W$ and viceversa. \label{item:P:4}
	\end{enumerate}
\end{enumerate} 
\end{theorem}

The above theorem has several interesting consequences.
For instance, if $V$ and $W$ are assumed to be geometrically nondegenerate\footnote{ This is the weakest nondegeneracy condition one usually considers, see Section \ref{subsec:nondeg} below.}, then the sum $X=V+W$ is known to be an ample divisor and so it is (very) singular unless $\deg(f)$ is sufficiently large and $V$ and $W$ are actually nondegenerate, which is a stronger nondegeneracy condition, see Corollary \ref{cor:singular} below.
The lower bound on $\deg(f)$ that we prove in item (\ref{item:P:2}) is optimal and coincides exactly with the degree of the addition map in the examples (\ref{item:JC}) and (\ref{item:JY}): 
$$
W_d(C)+W_{g-d-1}(C)=\Theta_C\ \ \text{and}\ \ F-F=\Theta_Y .
$$
By item (\ref{item:P:3:0}), the addition map achieves the minimal possible degree only in decompositions of theta divisors, and so, conjecturally, only in the two examples above.
Note also that item (\ref{item:P:3}) in Theorem \ref{thm:P} recovers precisely the formula of the geometric genus of the known examples of subvarieties of minimal class.

Property $(\mathcal P)$ in item (\ref{item:P:4}) of Theorem \ref{thm:P} is defined in Debarre's paper \cite[Section 2]{debarre} and goes back to Ran \cite{ran,ran2}; we recall the definition in Section \ref{subsec:propP} below. 
Although of a slightly technical nature, it is a very useful notion which frequently forces $V$ or $W$ to decompose further into a sum of lower dimensional subvarieties, cf.\ \cite{CMPS,debarre}.

Our original motivation for Theorem \ref{thm:P} comes from the study of generic vanishing (GV) subschemes, that is, of closed subschemes $Z$ of a ppav $(A,\Theta)$ whose  twisted ideal sheaf $\mathcal I_Z(\Theta)$ is a GV-sheaf.
Concretely, this means 
\[
\codim\{ L\in \Pic^0(A)\mid H^i(A,\mathcal I_Z(\Theta)\otimes L)\neq 0 \}\geq i \ \ \text{for all}\ \ i\geq 0 ;
\] 
it is a natural regularity condition for sheaves on abelian varieties \cite{pareschi-popa-JAMS,pareschi-popa-AJM}. 
For a nondegenerate subscheme $Y\subset \CP^n$ of projective space, the formal analogue of the above condition is $2$-regularity of $\mathcal I_{Y}$ in the sense of Castelnuovo--Mumford, see \cite{pareschi-popa-JAMS} and \cite[\S 2(7)]{pareschi-popa3}.
The latter is known to be equivalent to $Y$ being of minimal degree in $\CP^n$, which explains the expectation that geometrically nondegenerate GV-subschemes of ppavs should be precisely those subschemes whose cohomology classes are minimal. 
One direction of this conjecture had been proven by Pareschi and Popa \cite{pareschi-popa}: any geometrically nondegenerate GV-subscheme of a ppav has minimal cohomology class.
Conversely, the examples (\ref{item:JC}) and (\ref{item:JY}) of minimal class subvarieties are known to be GV-subschemes, see \cite{pareschi-popa-JAMS,pareschi-popa} and \cite{hoering}.
The generic vanishing conjecture of Pareschi and Popa \cite{pareschi-popa} predicts that these are the only examples of geometrically nondegenerate GV-subschemes on $g$-dimensional indecomposable ppavs of dimension $1\leq d\leq g-2$.
 
An important duality result of Pareschi and Popa associates to any geometrically non-degenerate GV-subscheme $Z$  of an indecomposable ppav $(A,\Theta)$ another GV-subscheme $V(Z)$, called the theta dual of $Z$, such that the reduced schemes $Z^{\red}$ and $-V(Z)^{\red}$ contain components $V$ and $W$ with $\Theta=V+W$.
As an application of Theorem \ref{thm:P}, we prove that $Z$ and $V(Z)$ are integral.
In particular, $V=Z$ and $W=-V(Z)$ have minimal cohomology classes and so $\Theta$ decomposes into a sum of minimal class subvarieties.  
We further deduce that $Z$ has the expected geometric genus and property $(\mathcal P)$ with respect to $-V(Z)$.
This result is important for applications; for instance, it plays a central role in the recent proof of the generic vanishing conjecture in dimension five, established in joint work of the author with Casalaina-Martin and Popa \cite{CMPS}.

\begin{theorem} \label{thm:GV} 
Let $(A,\Theta)$ be an indecomposable ppav of dimension $g$ and let $Z\subseteq A$ be a geometrically nondegenerate closed GV-subscheme of dimension $d$.
Then,
\begin{enumerate}
\item $Z$ and its theta dual $V(Z)$ are reduced and irreducible; in particular, $\Theta$ decomposes into a sum of minimal class subvarieties 
 of dimensions $d$ and $g-d-1$;\label{item:GV:1}  
\item the geometric genera are given by $p_g( Z)=\binom{g}{d}$ and  $p_g(V(Z))=\binom{g}{d+1}$;  \label{item:GV:2} 
\item $Z$ has property $(\mathcal P)$ with respect to $-V(Z)$ and viceversa.  \label{item:GV:4}
\end{enumerate}
\end{theorem} 


Theorems \ref{thm:P} and \ref{thm:GV} are related to a theorem of Ran  \cite{ran,ran2}. 
He proved that a $d$-dimensional nondegenerate subvariety $V$ of an abelian variety $A$ has property $(\mathcal P)$ with respect to a $(g-d)$-dimensional nondegenerate subvariety $W\subseteq A$ if the intersection number $V.W$ attains the smallest possible value among all nondegenerate subvarieties, which he proves to be $\binom{g}{d}$, see  \cite[Theorem 3.1]{debarre}.
This result was the main tool of Debarre's aforementioned proof of the minimal class conjecture for Jacobians in \cite{debarre}.
Items (\ref{item:P:2}) and (\ref{item:P:4})  in Theorem \ref{thm:P} are analogues of Ran's result under quite different assumptions. 
In contrast to Ran's theorem, item (\ref{item:P:1})  in Theorem \ref{thm:P} is not an assumption but a consequence.
Moreover, items (\ref{item:P:3:0}) and (\ref{item:P:3}) are new observations for which no analogues in Ran's situation are known (although it is natural to expect them).

\textbf{Notation and Conventions.}
We work over the field of complex numbers.
A variety is an integral separated scheme of finite type over $\C$.  
If $V$ and $W$ are subschemes of an abelian variety $A$, then we denote by $V+W$ the scheme theoretic image of the addition morphism $V\times W\longrightarrow A$. 

\section{Preliminaries}

\subsection{The trace map} \label{subsec:trace}
For any proper generically finite morphism $f:X\longrightarrow Y$ between complex varieties, there is a trace map
$$
f_{\ast}:H^{0}(X^{\sm},\Omega_{ X^{\sm}}^k)\longrightarrow H^{0}(Y^{\sm},\Omega_{Y^{\sm}}^k) ,
$$
where $X^{\sm}\subseteq X$ and $Y^{\sm}\subseteq Y$ denote the smooth loci, see \cite{griffiths}.
If $y\in Y$ is general with $f^{-1}(y)=\{x_1,\dots ,x_n\}$, then 
$$
(f_{\ast}\omega)_y= \sum_{i=1}^n \omega_{x_i},
$$
where we use the isomorphism $\Omega_{X,x_i}^k\cong \Omega_{Y,y}^k$, induced by $f$. 
If $X$ and $Y$ are smooth and proper, then the above trace map coincides with the restriction of the Gysin morphism
$$
f_{\ast}:  H^{k}( X,\C)\longrightarrow H^{k}( Y,\C)  ,
$$
to the subspaces of Hodge type $(k,0)$, see \cite[Section 7.3.2]{voisin1}.

\subsection{The Pontryagin product} \label{subsec:star}
For two cohomology classes $\alpha\in H^a(A,\C)$ and $\beta \in H^b(A,\C)$ on an abelian variety $A$, the Pontryagin product $\alpha\star \beta $ is defined via 
$$
\alpha\star \beta:=m_{\ast}(\pr_1^{\ast} \alpha \cup \pr_2^\ast \beta) \in H^{a+b-2g}(A,\C) ,
$$
where $m_{\ast}$ denotes the Gysin morphism (cf. \cite[\S 7.3.2]{voisin1}) with respect to the multiplication map $m:A\times A\longrightarrow A$. 
For instance, if $V$ and $W$ are subvarieties of $A$, then
$$
[V]\star [W]=\deg(f)\cdot [V+W],
$$ 
where $f:V\times W\longrightarrow V+W$ denotes the addition morphism, and where we put $\deg(f)=0$ if $f$ is not generically finite.

If $\hat A=\Pic^0(A)$ denotes the dual abelian variety, then there are natural isomorphisms $H^{2g-i}(\hat A,\C) \cong H^{2g-i}(A,\C)^{\ast}$ and so we obtain isomorphisms
$$
\PD:H^{i}(A,\C)\stackrel{\sim} \longrightarrow H^{2g-i}(\hat A,\C)  ,
$$ 
induced by $\omega \mapsto \int_A \omega\cup -$ and Poincar\'e duality. 
These isomorphisms  are compatible with the respective Hodge decompositions and exchange the Pontryagin product on $A$ with the cup product on $\hat A$, 
\begin{align} \label{eq:PD}
\PD(\alpha\star \beta)= \PD(\alpha)\cup \PD(\beta) ,
\end{align}
see for instance \cite[Proposition 1.7]{debarre-ein-laz-voisin}. 


\subsection{The property $(\mathcal P)$} \label{subsec:propP}
Let $V,W\subset A$ be subvarieties of an abelian variety, and let $f:V\times W\longrightarrow V+W$ denote the addition morphism.
Following Debarre \cite[Section 2]{debarre}, we say that $V$ has property $(\mathcal P)$ with respect to $W$ if for general $v\in V$, $v\times W$ is the only subvariety of $f^{-1}(v+W)$ which dominates $W$ via the second projection and $v+W$ via $f$. 
A guiding example \cite[Example 2.2]{debarre} is given by the Brill--Noether locus $W_d(C)$ inside the Jacobian of a smooth projective curve $C$ of genus $g$, which has property $(\mathcal P)$ with respect to $W_{g-e}(C)$ for all $e\geq d$.

Our interest in this property arises from  \cite{schreieder}, where we have shown that an indecomposable ppav $(A,\Theta)$ is isomorphic to the Jacobian of a smooth curve if and only if $\Theta=C+W$ for some curve $C\subset A$.
On the other hand, the property $(\mathcal P)$ is a very efficient tool for producing curve summands.  
The following example shows a baby case of this phenomenon; more elaborate arguments are contained in \cite{CMPS} and \cite{debarre}.

\begin{example} \label{ex:propP}
Let $C\subset A$ be a curve in an abelian variety which has property $(\mathcal P)$ with respect to another subvariety $W\subset A$.
If the addition morphism $f:C\times W\longrightarrow C+W$ is not birational, then $W=C+W'$ for some subvariety $W'\subset A$. 
\end{example}
\begin{proof}
Let $c\in C$ be a general point. 
Then the reduced preimage of $c+W$ decomposes as
$$
f^{-1}(c+W)^{\red}=c\times W\cup R \cup Q ,
$$
with $f(Q)\subsetneq c+W$, and such that each component of $R$ dominates $c+W$ via $f$.
Since $f$ is not birational, $R\neq \emptyset$ and so we can pick a component $R'$ of $R$.
Since $C$ has property $(\mathcal P)$ with respect to $W$, $\pr_2(R')\subsetneq W$, which implies
$
R'=C\times \pr_2(R') 
$.
Applying $f$ shows then that $W$ has a curve summand, as we want.
\end{proof}

\subsection{Geometrically nondegenerate subschemes} \label{subsec:nondeg}
Let $Z\subseteq A$ be a closed subscheme of dimension $d$ of a $g$-dimensional abelian variety $A$.
Following Ran \cite{ran}, $Z$ is called \textit{nondegenerate} if the cup product map
$$
\cup [Z]:H^{d,0}(A)\longrightarrow H^{g,g-d}(A)
$$
is injective (hence an isomorphism).  
This definition does not depend on the components $Z'$ of $Z$ of dimension $<d$, as such components satisfy $[Z']\cup \omega=0$ for all $\omega\in H^{d,0}(A)$.
If $Z$ is reduced and equi-dimensional, it is nondegenerate if and only if the image of the Gauss map $G_Z:Z\dashrightarrow \Gr(d,g)$ is via the Pl\"ucker embedding not contained in any hyperplane, see \cite[Section II]{ran}. 

We say that $Z$ is \textit{geometrically nondegenerate} if the kernel of the above cup product map contains no decomposable elements, i.e. nontrivial elements of the form $\alpha_1\cup \dots \cup \alpha_d$  with  $\alpha_i\in H^{1,0}(A)$. 
A slightly different definition was previously given by Ran \cite{ran}, who assumes that $Z$ is pure-dimensional and reduced, and asks that any nonzero decomposable holomorphic $d$-form on $A$ pulls back to a nonzero form on some resolution of singularities of $Z$.
The following lemma shows that both notions of geometric nondegeneracy coincide.
Moreover, $Z$ is geometrically nondegenerate if and only if the reduced scheme $Z^{\red}$ is.

\begin{lemma} \label{lem:geomnondeg}
Let $Z$ be a closed subscheme of dimension $d$ of a $g$-dimensional abelian variety $A$.
Then, $Z$ is geometrically nondegenerate if and only if each nonzero decomposable holomorphic $d$-form on $A$ pulls back to a nonzero form on some resolution of singularities $\widetilde Z^{\red}$ of the reduced scheme $Z^{\red}$.
\end{lemma}

\begin{proof}   
We follow the arguments in \cite[Lemma II.1]{ran}, where a similar statement is proven for nondegenerate reduced subschemes.

Since any holomorphic $d$-form vanishes on a variety of dimension $<d$, we may without loss of generality assume that $Z$ has pure dimension $d$ and so $[Z]\in H^{2g-2d}(A)$.
Hence, $[Z]=\sum _i a_i [Z_i]$, where $Z_i$ runs through the irreducible components of the reduced scheme $Z^{\red}$, and $a_i$ denotes the multiplicity of $Z_i$ in $Z$.

Let $\omega\in H^{d,0}(A)$ be a decomposable holomorphic $d$-form.
If $Z$ is geometrically nondegenerate, then 
$$
0\neq \omega \cup [Z] =\sum_i a_i \cdot \omega \cup [Z_i] .
$$
Let $j_i:\widetilde {Z}_i\longrightarrow A$ be the composition of a resolution of singularities $\widetilde {Z}_i\longrightarrow Z_i$ with the inclusion $Z_i\subseteq A$.
Then, 
$$
\omega \cup [Z_i] ={j_i}_{\ast}\circ {j_i}^{\ast} ( \omega) ,
$$
which proves that there is a component $Z_i$ with ${j_i}^{\ast} ( \omega)\neq 0$ and so $\omega$ pulls back to a nontrivial class on $\widetilde Z^{\red}$.

Conversely, if $\omega$ is a decomposable holomorphic $d$-form on $A$ with ${j_0}^{\ast} ( \omega)\neq 0$ for some component $Z_{0}$ of $Z^{\red}$, then, by the Hodge--Riemann bilinear relations,
$$
0< \epsilon \cdot  \overline {{j_0}^{\ast} ( \omega)} \cup {j_0}^{\ast} ( \omega)= \epsilon \cdot \overline{\omega}\cup \omega \cup [Z_0] ,
$$
where $\epsilon$ is a constant which depends only on $d$.
Similarly, $0\leq \epsilon \cdot \overline{\omega}\cup \omega \cup [Z_i] $ for all $i$ and so  
$$
0< \sum_i \epsilon \cdot \overline{\omega}\cup \omega \cup a_i[Z_i] =\epsilon \cdot \overline {\omega} \cup \omega \cup [Z].
$$
Hence, $Z$ is geometrically nondegenerate, as we want.
\end{proof}

The following lemma describes an important property of geometrically nondegenerate subschemes.
If $Z$ is reduced and irreducible, it is due to Debarre  \cite[p.\ 105]{debarre-tores}.

\begin{lemma} \label{lem:sum}
Let $Z$ be a closed geometrically nondegenerate subscheme of dimension $d$ of a $g$-dimensional abelian variety $A$.
Then, for any closed subscheme $Z'\subseteq A$, we have $\dim(Z+Z')=\dim(Z)+\dim(Z')$ or $Z+Z'=A$.
\end{lemma}

\begin{proof} 
Replacing $Z'$ by a component of maximal dimension of the reduced scheme ${Z'}^{\red}$, we may assume that $Z'$ is a $d'$-dimensional subvariety of $A$.
Let $z'\in Z'$ be a smooth point and consider the tangent space $T_{Z',z'}$, which we think of as subspace of $T_{A,0}$.
Let $L\subseteq T_{A,0}$ be a $(g-d)$-dimensional subspace which contains $T_{Z',z'}$ if $d'\leq g-d$ and which is contained in $T_{Z',z'}$ if $d'>g-d$.
This subspace gives rise to a decomposable holomorphic $d$-form $\omega$ on $A$, given by the quotient map 
$$
\Lambda^dT_{A,0}\longrightarrow \Lambda^dT_{A,0}\slash (L\wedge \Lambda^{d-1}T_{A,0}) \cong \C ,
$$
well-defined up to a nonzero multiple.
By Lemma \ref{lem:geomnondeg}, there is a  component $Z_0$ of the reduced scheme $Z^{\red}$ such that $\omega$ pulls back to a nonzero form on some resolution of singularities of $Z_0$.
This implies that $Z_0$ has dimension $d$ and that for a general point $z_0\in Z_0$, $T_{Z_0,z_0}\cap L=0$.
Therefore, $T_{Z_0,z_0}$ meets $T_{Z',z'}$ transversely, where we think of both vector spaces as subspaces of $T_{A,0}$.
This proves the lemma, because the differential of the addition morphism $Z_0\times Z'\longrightarrow Z_0+Z'$ at $(z_0,z')$ is given by vector addition.
\end{proof}

\begin{example} \label{ex:ample}
A divisor $D\subset A$ on an abelian variety $A$ is geometrically nondegenerate if and only if it is ample.
\end{example}
\begin{proof}
If $D$ is ample, then, by the Hard Lefschetz Theorem, $D$ is nondegenerate, hence geometrically nondegenerate.
Conversely, if $D$ is geometrically nondegenerate and $C\subset A$ is a curve, then $D-C= A$ by Lemma \ref{lem:sum} and so a general translate of $C$ meets $D$ in a positive number of points.
Hence, $D$ is ample, as we want.
\end{proof}

Debarre proved that the sum of geometrically nondegenerate subvarieties is geometrically nondegenerate \cite[p.\ 105]{debarre-tores}.
Example \ref{ex:ample} has therefore the following consequence.

\begin{corollary} \label{cor:V+W=ample}
Let $A$ be a $g$-dimensional abelian variety, and let $V,W\subseteq A$ be closed geometrically nondegenerate subvarieties of dimensions $d$ and $g-1-d$, respectively.
Then, $X=V+W$ is an ample divisor on $A$. 
\end{corollary}


%
%

\subsection{Generic vanishing subschemes} \label{subsec:pareschi-popa}
Let $(A,\Theta)$ be a $g$-dimensional ppav and let $Z\subseteq A$ be a closed subscheme. 
Following Pareschi and Popa \cite{pareschi-popa}, we say that $Z$ is a GV-subscheme of $A$, if the twisted ideal sheaf $\mathcal I_Z(\Theta)$ is a GV-sheaf. 
By \cite[Theorem 2.1]{pareschi-popa}, this means that the complex
\begin{align*} 
\R\hat{\mathcal S}(\R {\mathcal Hom}(\mathcal I_{Z}(\Theta),\mathcal O_A)) \in D^b(\hat A)
\end{align*}
in the derived category of the dual abelian variety $\hat A$ has zero cohomology in all degrees $i\neq g$. 
Here, $\R\hat{\mathcal S}:\D^b(A)\longrightarrow \D^b(\hat A)$ denotes the Fourier--Mukai transform with respect to the Poincar\'e line bundle. 

The theta dual $V(Z)\subseteq A$ of a subscheme $Z\subset A$ is the scheme theoretic support of 
$$
(-1_{\hat A})^{\ast}R^g \hat{\mathcal S}(\R {\mathcal Hom}(\mathcal I_{Z}(\Theta),\mathcal O_A)) ,
$$
where we use $\Theta$ to identify $\hat A$ with $A$.
Set theoretically, $V(Z)=\{x\in A\mid Z-x\subseteq \Theta\}$, see \cite[p.\ 216]{pareschi-popa}. 
If $Z$ is geometrically nondegenerate of dimension $d$, then $\dim(V(Z))\leq g-d-1$ by Lemma \ref{lem:sum}.
Moreover, if $\dim(V(Z))= g-d-1$ and $\Theta$ is irreducible, then the reduced schemes $Z^{\red}$ and $-V(Z)^{\red}$ contain components $V$ and $W$ with $\Theta=V+W$.

In \cite{pareschi-popa}, Pareschi and Popa state their theorems only for geometrically nondegenerate reduced GV-subschemes of pure dimension. 
However, the same proofs work without the reducedness and pure-dimensionality assumptions. 

\begin{theorem}[Pareschi--Popa \cite{pareschi-popa}] \label{thm:pareschi-popa}
Let $(A,\Theta)$ be a $g$-dimensional ppav and let $Z\subseteq A$ be a geometrically nondegenerate closed GV-subscheme of dimension $d$. 
Then,
\begin{enumerate}
\item $Z$ and $V(Z)$ are pure-dimensional Cohen--Macauly subschemes.  
\item $V(Z)$ is a $(g-d-1)$-dimensional GV-subscheme with $V(V(Z))=Z$.
\item $Z$ and $V(Z)$ have minimal cohomology classes $[Z]=\frac{\theta^{g-d}}{(g-d)!}$ and $[V(Z)]=\frac{\theta^{d+1}}{(d+1)!}$.
\end{enumerate}
\end{theorem}

\begin{proof} 
Since $Z$ is geometrically nondegenerate, $\dim(V(Z))\leq g-d-1$ by Lemma \ref{lem:sum} above, and so Lemma 4.4 in \cite{pareschi-popa} remains true under our assumptions.
Therefore, the arguments in \cite[Theorem 5.2]{pareschi-popa} prove that the theta dual $V(Z)\subseteq A$ is a closed Cohen--Macauly subscheme of pure dimension $g-d-1$. 
Moreover, $V(Z)$ is a GV-subscheme  of $A$ with $V(V(Z))=Z$ and so $Z$ is also pure-dimensional and Cohen--Macauly.
(Geometric nondegeneracy of $V(Z)$ is not needed here, $\dim(Z)+\dim(V(Z))=g-1$ is enough.)
The final assertion follows as in \cite[Theorem 6.1]{pareschi-popa}. 
\end{proof}

\begin{remark}
Dropping the reducedness assumption on the GV-subschemes in \cite{pareschi-popa} seems to be necessary, because the theta dual of a reduced GV-subscheme might a priori be nonreduced. 
\end{remark}

\begin{remark}
Theta duality has further been investigated by Gulbrandsen--Lahoz \cite{gulbrandsen-lahoz}; embeddings of GV-subschemes in ppavs have been studied by Lombardi--Tirabassi \cite{lombardi-tirabassi}. 
\end{remark}

\section{Endomorphisms attached to two subvarieties of an abelian variety} \label{sec:P} 
Let $V$ and $W$ be closed subvarieties of an abelian variety $A$. 
Suppose that the sum $X:=V+W$ has the expected dimension $\dim(X)=\dim(V)+\dim(W)$.
Then the addition morphism $f:V\times W\longrightarrow X$ is a generically finite map, and so, for a general point $x\in X$,
\[
f^{-1}(x)=\left\{(v_1,w_1),\dots ,(v_s,w_s)\right\} ,
\]
where $s:=\deg(f)$.
Since $x\in X$ is general, it is a smooth point of $X$ and $f$ is unramified above $x$.
Therefore, 
\[
T_{V,v_i}\oplus T_{W,w_i}=T_{X,x}
\]
for all $i=1,\ldots ,s$.
This decomposition gives rise to a projector
\[
P_i:T_{X,x} \longrightarrow T_{X,x}
\]
with kernel $T_{W,w_i}$ and image $T_{V,v_i}$.
Taking the $k$-th exterior power yields an endomorphism
\begin{align} \label{def:c(V,W)}
c_k(V,W)(x):=\sum_{i=1}^s\Lambda^kP_i: \Lambda^kT_{X,x} \longrightarrow \Lambda^kT_{X,x} .
\end{align} 

\begin{remark}
The endomorphism in (\ref{def:c(V,W)}) generalizes Ran's work \cite[Section I.1]{ran}.
He considered complementary dimensional subvarieties $V,W\subset A$, which meet transversely in smooth points, and constructed an endomorphism of $\Lambda ^k T_{A,0}$ which coincides (under these assumptions) with $c_k(V,-W)(0)$ above.
\end{remark}

Somewhat surprisingly, the following result shows that the above endomorphism can be computed quite explicitly in the case where $X$ is an ample divisor with only rational singularities.
This result is the key step in the proof of Theorem \ref{thm:P}.

\begin{theorem} \label{thm:prop:P} 
Let $A$ be a $g$-dimensional abelian variety and let $V,W\subset A$ be closed subvarieties of respective positive dimensions $d$ and $g-1-d$.
If $X=V+W$ is an ample divisor on $A$ with only rational singularities, then, for all $1\leq k\leq d$,  
$$
c_k(V,W)(x)=\lambda_k\cdot \id ,
$$
for some nonzero constant $\lambda_k\in \Q^{\times}$ which does not depend on $x$.  
\end{theorem}

The remainder of this section is devoted to the proof of the above theorem; the notation will always be that of the theorem.

\subsection{Singularities of $X$}
The assumptions on the singularities of $X$ imply the following well-known result, where $j:\widetilde X \longrightarrow A$ denotes the composition of a resolution of singularities $r:\widetilde X\longrightarrow X$ with the inclusion $h:X\longrightarrow A$.

\begin{lemma} \label{lem:Hk0X} 
The pullback map $j^\ast: H^{0}(A,\Omega_A^k)\longrightarrow H^0(\widetilde X,\Omega_{\widetilde X}^k)$ is an isomorphism for all $k\leq g-2$ and injective for $k=g-1$.
Moreover, $j^{\ast}$ is surjective for $k=g-1$ if and only if $h^{0}(A,\mathcal O_A(X))=1$.
\end{lemma}

\begin{proof}
Since $X$ has only rational singularities, it is normal and $R^i r_\ast \mathcal O_{\widetilde X}=0$ for $i\geq 1$.
Using the Leray spectral sequence, it follows that the pullback map 
\begin{align} \label{eq:g*}
r^{\ast}:H^k(X,\mathcal O_X)\stackrel{\sim}\longrightarrow  H^k(\widetilde X,\mathcal O_{\widetilde X}) 
\end{align} 
is an isomorphism for all $k$.

We have the following short exact sequence
$$
0\longrightarrow \mathcal O_A(-X)\longrightarrow \mathcal O_A \longrightarrow \mathcal O_X \longrightarrow 0 .
$$
By Kodaira vanishing, the restriction map $H^k(A,\mathcal O_A)\longrightarrow H^k(X,\mathcal O_X)$ is an isomorphism for all $k\leq g-2$ and injective for $k=g-1$.
Together with (\ref{eq:g*}), this proves that
$$
j^{\ast}: H^k(A,\mathcal O_A)\stackrel{\sim}\longrightarrow H^k(\widetilde X,\mathcal O_{\widetilde X})
$$
is an isomorphism for all $k\leq g-2$ and injective for $k=g-1$. 
Moreover, $j^{\ast}$ is surjective for $k=g-1$ if and only if the surjection
$
H^{g}(A,\mathcal O_A(-X))\longrightarrow H^g(A,\mathcal O_A)
$
is injective, which by Serre duality is equivalent to  $h^{0}(A,\mathcal O_A(X))=1$.
The lemma follows now via complex conjugation from the Hodge decomposition theorem. 
\end{proof}

\subsection{Construction of a useful correspondence}
Let $\widetilde V\longrightarrow V$ and $\widetilde W\longrightarrow W$  be resolutions of singularities, and let $i:\widetilde V\longrightarrow A$ denote the composition of the resolution $\widetilde V\longrightarrow V$  with the embedding of $V$ in $A$.

We consider the following commutative diagram
\begin{align} \label{diag:mu}
\begin{xy}
  \xymatrix{
A& \ar[l]^{i} \widetilde V		 &  \ar[l]^{\pr_1}\widetilde V\times \widetilde W \ar[d]^{\mu}   &  	\ar[l]^{\tilde r} 	\widetilde{ V\times W}\ar[d]^{\tilde \mu}	\\
 &	A		&  \ar[l]^{h}			X 							& \ar[l]^r		\widetilde{X}	.
  }	
\end{xy}
\end{align}
Here, $\mu$ is the composition of $\widetilde V\times \widetilde W\longrightarrow V\times W$ with the addition morphism $f:V\times W\longrightarrow X$, and $\tilde r$ is a sequence of blow-ups along smooth centers such that $\mu$ admits a lift $\tilde \mu: \widetilde{ V\times W}\longrightarrow \widetilde X $, where $r:\widetilde X\longrightarrow X$ is the resolution of singularities of $X$ from above. 

Denoting by $q:	\widetilde{ V\times W} \longrightarrow A$ the composition of $\tilde r$, $\pr_1$ and $i$ in the first row of (\ref{diag:mu}), we obtain the following short version of (\ref{diag:mu}):
\begin{align*} 
  \xymatrix{
A& \ar[l]^{q} 	\widetilde{ V\times W}\ar[d]^{\tilde \mu}	\\
	&			\widetilde{X}	.
  }
\end{align*}
The main idea of the proof of Theorem \ref{thm:prop:P} is to study the homomorphism 
$$
\tilde \mu_\ast\circ q^\ast:H^k(A,\C)\longrightarrow H^k(\widetilde X,\C)
$$ 
on a purely topological level and to compare this with the result one obtains by studying the above map on the level of holomorphic forms. 
This approach is motivated by Ran's work \cite{ran}, but additional difficulties arise in our situation.
For instance, due to possible singularities of $X$, we have in general no control over $H^k(\widetilde X,\C)$ and so it is hard to compute $\tilde \mu_{\ast}\circ q^{\ast}$ directly. 
However, once we apply $j_\ast$, a concrete description is given by Lemma \ref{lem:ast} below.

\subsection{Topological computations} \label{subsec:topo} 
The following lemma computes  $j_\ast\circ \tilde \mu_\ast\circ q^\ast$ in terms of the Pontryagin product $\star$, cf.\ Section \ref{subsec:star}.

\begin{lemma} \label{lem:ast}
The map $j_\ast\circ \tilde \mu_\ast\circ q^\ast:H^k(A,\C)\longrightarrow H^{k+2}(A,\C)$ is given by
\[
j_\ast\circ \tilde \mu_\ast\circ q^\ast(\alpha)= (\alpha\cup [V]) \star [W] .
\]
\end{lemma}

\begin{proof} 
By construction, $\tilde r:\widetilde{V\times W}\longrightarrow \widetilde V\times \widetilde W$ is a sequence of blow-ups along smooth centers.
By the formula for the cohomology of such blow-ups \cite[p.\ 180]{voisin1},
\[
\tilde r_\ast\circ \tilde r^\ast:H^k( \widetilde V\times \widetilde W,\C)\longrightarrow H^k( \widetilde V\times \widetilde W,\C) 
\]
is the identity.
Using the commutativity of the right square in (\ref{diag:mu}), this yields  
\begin{align} \label{eq:mu}
 r_\ast \circ \tilde \mu_\ast \circ \tilde r^\ast= \mu_\ast \circ \tilde r_\ast \circ   \tilde r^\ast=\mu_\ast .
\end{align}
Since $q^\ast=\widetilde r^{\ast}\circ \pr_1^{\ast} \circ i^{\ast}$ and $j_{\ast}=h_{\ast}\circ r_{\ast}$, this implies  
\begin{align} \label{eq:lem:ast}
j_\ast\circ \tilde \mu_\ast\circ q^\ast = h_\ast\circ \mu_\ast \circ \pr_1^\ast\circ i^\ast :H^k(A,\C)\longrightarrow H^{k+2}(A,\C) .
\end{align}

We denote the Poincar\'e duality isomorphisms on $A$, $\widetilde V$ and $\widetilde V\times \widetilde W$, given by cap product with the corresponding fundamental classes, by $\operatorname{D}_A$, $\operatorname{D}_{\widetilde V}$ and $\operatorname{D}_{\widetilde V\times \widetilde W}$ respectively. 
Let $\alpha\in H^k(A,\C)$ and consider the pullback $\pr_1^\ast ( i^\ast \alpha)$ to $\widetilde V\times \widetilde W$. 
Then,
\begin{align} \label{eq:D_VxW}
\operatorname{D}_{\widetilde V\times \widetilde W}(\pr_1^\ast ( i^\ast \alpha)) =\operatorname{D}_{\widetilde V}(i^\ast \alpha) \otimes [\widetilde W]\in H_{2d-2k}(\widetilde V, \C)\otimes H_{2g-2d-2}(\widetilde W, \C), 
\end{align}
where we use the cross product on homology with coefficients in a field to identify $ H_{2d-2k}(\widetilde V, \C)\otimes H_{2g-2d-2}(\widetilde W, \C)$ with a direct summand of $H_{2g-2-k}(\widetilde V\times \widetilde W,\C)$. 
The pushforward of the above homology class to $A$ is given by
\begin{align*}
h_\ast\circ \mu_\ast (\operatorname{D}_{\widetilde V}(i^\ast \alpha) \otimes [\widetilde W])&=(i_{\ast}\operatorname{D}_{\widetilde V}(i^\ast \alpha)) \star [W] \in H_{2g-2-k}(A,\C) .
\end{align*} 
Using (\ref{eq:lem:ast}) and (\ref{eq:D_VxW}), this proves by the definition of the Gysin morphisms $(h\circ \mu)_{\ast}$ and $i_{\ast}$ (cf. \cite[\S 7.3.2]{voisin1}),
\begin{align*}
j_\ast\circ \tilde \mu_\ast\circ q^\ast (\alpha) 
&=(h\circ \mu)_\ast (\pr_1^\ast ( i^\ast(\alpha))) \\
&=\operatorname{D}_A( (h\circ \mu)_{\ast} (\operatorname D_{\widetilde V\times \widetilde W} (\pr_1^\ast(i^\ast\alpha) ) ) ) \\
&=\operatorname{D}_A( (h \circ \mu)_{\ast} ( \operatorname{D}_{\widetilde V}(i^\ast \alpha) \otimes [\widetilde W] ) ) \\
&= \operatorname{D}_{A}(i_{\ast}\operatorname{D}_{\widetilde V}(i^\ast \alpha)) \star [W] \\
&= (i_{\ast}i^{\ast}(\alpha)) \star [W] \\
&= (\alpha\cup [V] )\star [W].
\end{align*}
This finishes the proof of the lemma.
\end{proof}

\begin{remark}
If we assume that $V$ and $W$ are non-degenerate, then it follows easily from (\ref{eq:PD}) that the  homomorphism from Lemma \ref{lem:ast} is nonzero. 
The following lemma shows that the same statement holds without additional assumptions on $V$ or $W$.
\end{remark}

\begin{lemma} \label{lem:nonzero}
For $0\leq k\leq d$, there is a class $\omega\in H^{k,0}(A)$ with $(\omega\cup [V]) \star [W]\neq 0$.
\end{lemma}

\begin{proof} 
Recall from Section \ref{subsec:star} the isomorphisms 
$
\PD:H^{i}(A,\C)\stackrel{\sim} \longrightarrow H^{2g-i}(\hat A,\C)
$, 
induced by $\omega \mapsto \int_A \omega\cup -$ and Poincar\'e duality.
By (\ref{eq:PD}), $\PD$ exchanges the Pontryagin and the cup product.
It therefore suffices to find a class $\omega\in H^{k,0}(A)$ with
\[
\PD(\omega \cup [V])\cup \PD([W])\neq 0 .
\] 
Since the addition morphism $V\times W\longrightarrow X$ is generically finite, $[V]\star[W]\neq 0$ and so $\PD([V])\cup \PD([W])\neq 0$.
It therefore suffices to prove that there are classes $\omega_i\in H^{k,0}(A)$ and $\omega_i'\in H^{0,k}(\hat A)$ with
\begin{align} \label{eq:=PD(V)}
\sum_i \omega_i'\cup \PD(\omega_i \cup [V])=\PD([V]) .
\end{align}

The existence of suitable classes $\omega_i$ and $\omega_i'$ which satisfy the above identity is established by a straightforward but somewhat lengthy calculation.
To begin with, let $A=\C^g\slash \Gamma$ and let $z_1,\ldots ,z_g$ be coordinates on $\C^g$.
These coordinates give rise to a basis 
$
dz_1,\ldots ,dz_g,d\overline z_1,\ldots ,d\overline z_g
$
 of $H^{1}(A,\C)$.
The corresponding dual basis 
$$
dz_1^\ast,\ldots ,dz_g^\ast,d\overline z_1^\ast,\ldots ,d\overline z_g^\ast
$$
can be identified with a basis of $H^1(\hat A,\C)$, but note that $dz_{i}^\ast \in H^{0,1}(\hat A)$ and $d\overline{z}_j^{\ast} \in H^{1,0}(\hat A)$.

Since $H^k(X,\C)=\Lambda^kH^1(X,\C)$, the above basis of $H^1(X,\C)$ gives rise to a basis 
$$
dz_I\cup d\overline z_{J}=dz_{i_1}\cup\ldots \cup dz_{i_p}\cup d\overline z_{j_1}\cup \ldots \cup d\overline z_{j_q}
$$
of $H^{p,q}(A)$, where $I=(i_1,\ldots ,i_{p})$ and $J=(j_1,\ldots ,j_{q})$ run through all indices with $1\leq i_1 < i_2< \ldots <i_{p}\leq g$ and  $1\leq j_1 < j_2< \ldots <j_{q}\leq g$, respectively.
Similarly, 
$$
dz^{\ast}_I\cup d\overline z^{\ast}_{J}=dz^{\ast}_{i_1}\cup\ldots \cup dz^{\ast}_{i_p}\cup d\overline z^{\ast}_{j_1}\cup \ldots \cup d\overline z^{\ast}_{j_q} ,
$$
yields a basis of $H^{q,p}(\hat A)$, where $I$ and $J$ run through the same indices as above.

Using these basis elements, we have
$$
\PD(dz_I\cup d\overline z_J)=\epsilon_{I,J}\cdot  dz^{\ast}_{I^c}\cup d \overline z^{\ast}_{J^c} ,
$$
where $I\cup I^c=J\cup J^c=\{1,2,\ldots ,g\}$ and $\epsilon_{I,J}$ is determined by
$$
\epsilon_{I,J}=\PD(dz_I\cup d\overline z_J)(dz_{I^c}\cup d\overline z_{J^c})=\int_A dz_I\cup d\overline z_J\cup dz_{I^c}\cup d\overline z_{J^c} .
$$
In particular, only the sign of $\epsilon_{I,J}$ depends on $I$ and $J$.

Let us now return to the proof of (\ref{eq:=PD(V)}).
For suitable $\lambda_{I,J}\in \C$, we have
\begin{align*}
[V]=\sum_{I,J} \lambda_{I,J} \cdot dz_{I}\cup d\overline z_{J} ,
\end{align*}
where the sum runs over all indices $I$ and $J$ of length $g-d$.
Let $L=(l_1,\ldots ,l_k)$ with $1\leq l_1\leq \ldots \leq l_k\leq g$ be an index of length $k$.
Then
\begin{align*}
\PD\left(dz_{L} \cup [V]\right)&=\PD\left( \sum_{I,J \mid L\subseteq I^c} \lambda_{I,J} \cdot dz_{L} \cup dz_{I}\cup d\overline z_{J} \right) \\
&=\PD\left( \sum_{I,J \mid L\subseteq I^c} \lambda_{I,J} \cdot \delta_{L,I}\cdot dz_{L\cup I} \cup d\overline z_{J}\right) \\
&=\sum_{I, J\mid L\subseteq I^c} \lambda_{I,J} \cdot  \epsilon_{L\cup I,J} \cdot \delta_{L,I}\cdot dz_{I^c\setminus L}^{\ast}\cup d\overline z^{\ast}_{J^c} ,
\end{align*}
where the sum runs through all $I$ and $J$ of length $g-d$ such that additionally $L\subseteq I^c$.
Moreover, $L\cup I$ denotes the ordered tuple whose underlying set is the union of $L$ and $I$, and the sign $\delta_{L,I}$ can be computed from
$$
dz_{L} \cup dz_{I}=\delta_{L,I}\cdot dz_{L\cup I} .
$$ 
Using the above definitions, a careful sign check shows
$$
\epsilon_{L\cup I,J}\cdot \delta_{L,I}\cdot dz_{L}^{\ast}\cup dz_{I^c\setminus L}^{\ast}=\epsilon_{I,J} \cdot dz_{I^c}^{\ast} .
$$ 
Using this, we obtain
\begin{align*}
\sum_L dz_{L}^{\ast} \cup \PD(dz_{L} \cup [V])
& =\sum_{L}\sum_{I,J\mid L\subseteq I^c} \lambda_{I,J} \cdot \epsilon_{L\cup I,J} \cdot\delta_{L,I}\cdot dz_{L}^{\ast}\cup dz_{I^c\setminus L}^{\ast}\cup d\overline z_{J^c}^{\ast} \\
& = \sum_{L}\sum_{I,J\mid L\subseteq I^c}   \lambda_{I,J} \cdot \epsilon_{I,J} \cdot dz_{I^c}^{\ast}\cup d\overline z_{J^c}^{\ast} \\
& =  \sum_{I,J}\sum_{L \mid L\subseteq I^c}  \lambda_{I,J} \cdot \epsilon_{I,J} \cdot  dz_{I^c}^{\ast}\cup d\overline z_{J^c}^{\ast} \\
&=  \binom{d}{k} \cdot \sum_{I,J}  \lambda_{I,J} \cdot \epsilon_{I,J} \cdot  dz^{\ast}_{I^c}\cup d\overline z^{\ast}_{J^c} \\
&=  \binom{d}{k} \cdot \PD([V]) .
\end{align*}  
Since $0\leq k\leq d$, $\binom{d}{k}\neq 0$.
This proves (\ref{eq:=PD(V)}), which finishes the proof of the lemma. 
\end{proof}

\subsection{A description in terms of holomorphic forms}

Since $d\leq g-2$, Lemma \ref{lem:Hk0X} implies that
$$
j^\ast: H^{k,0}(A)\stackrel{\sim}\longrightarrow H^{k,0}(\widetilde X)
$$
is an isomorphism for all $1\leq k\leq d$.
Therefore, studying $\tilde \mu_\ast\circ q^\ast$ on the level of holomorphic forms is equivalent to studying the composition 
$$
\psi:=(j^\ast)^{-1}\circ \tilde \mu_\ast\circ q^\ast: H^{k,0}(A)\longrightarrow H^{k,0}(A) ,
$$
where $\tilde \mu_{\ast}$ denotes the trace map on holomorphic forms, see Section \ref{subsec:trace}.
 
For general $x\in X$, we recall from (\ref{def:c(V,W)}) the endomorphism $c_k(V,W)(x)$ of $\Lambda^kT_{X,x}$ and consider its transpose
\begin{align} \label{def:c(V,W)^t}
c_k^t(V,W)(x): \Omega^k_{X,x} \longrightarrow \Omega^k_{X,x} .
\end{align}
The key property of this endomorphism is uncovered by the following lemma.

\begin{lemma} \label{lem:psi}
Let $x\in X$ be a general point and let $\pi_x:H^{k,0}(A)\cong \Omega^k_{A,x}\longrightarrow \Omega^k_{X,x}$ be the natural restriction morphism.
Then the following diagram is commutative
\begin{align} \label{diag:H^{k,0}}
\begin{xy}
  \xymatrix{
H^{k,0}(A) \ar[rr]^{\psi}\ar[d]^{\pi_x}	&	&  \ar[d]^{\pi_x} H^{k,0}(A) 	\\
\Omega^k_{X,x} \ar[rr]^{ c^t_k(V,W)(x)} & & \Omega^k_{X,x} .
  }	
\end{xy}
\end{align} 
\end{lemma}
\begin{proof}
Since $x\in X$ is general, there is a unique (and general) point $\tilde x\in \widetilde X$ with $r(\tilde x)=x$ and we may assume that $r:\widetilde X\longrightarrow X$ is an isomorphism in a neighbourhood of $\tilde x$.
This induces an isomorphism
$$
r^{\ast} :\Omega_{X,x}^k\stackrel{\sim}\longrightarrow  \Omega_{\widetilde X,\tilde x}^k .
$$
Via this isomorphism, $c^t_k(V,W)(x)$ corresponds to an endomorphism
\[
c_k^t(V,W)(\tilde x):=r^{\ast}\circ c_k^t(V,W)(x) \circ (r^{\ast})^{-1} : \Omega^k_{\widetilde X,\tilde x}\longrightarrow \Omega^k_{\widetilde X,\tilde x} .
\] 

Since $\tilde x\in \widetilde X$ is general, $\tilde \mu^{-1}(\tilde x)$ lies in the locus where $\widetilde {V\times W}\longrightarrow V\times W$ is an isomorphism.
In particular, $\widetilde \mu^{-1}(\tilde x)$ can be identified with $f^{-1}(x)$.  
Therefore, comparing the definitions of the trace map $\tilde \mu_{\ast}$ and the endomorphism $c_k^t(V,W)(\tilde x)$, we obtain
\begin{align*} 
(\tilde \mu_\ast ( q^\ast \omega))_{\tilde x}=c_k^t(V,W)(\tilde x)((j^\ast \omega)_{\tilde x}) ,
\end{align*}
for all $\omega \in H^{k,0}(A)$, where we use that $\omega$ is translation invariant. 
The above identity holds in $\Omega^k_{\widetilde X,\tilde x}$, and via the isomorphism $(r^{\ast})^{-1}:  \Omega_{\widetilde X,\tilde x}^k \stackrel{\sim}\longrightarrow  \Omega_{X,x}^k$, we obtain 
$$
(\tilde \mu_\ast ( q^\ast \omega))_{\tilde x} \mapsto \pi_x(\psi(\omega)) ,
$$
$$
c_k^t(V,W)(\tilde x)((j^\ast \omega)_{\tilde x}) \mapsto c_k^t(V,W)( x)(\pi_x(\omega)).
$$
This proves the lemma. 
\end{proof}  

\begin{proposition}\label{prop:psi}
Let $1\leq k\leq d$.
Then,
$$
\psi=(j^\ast)^{-1}\circ \tilde \mu_\ast\circ q^\ast: H^{k,0}(A)\longrightarrow H^{k,0}(A)
$$ 
is a nonzero multiple of the identity.
\end{proposition}

\begin{proof}
Lemma \ref{lem:psi} implies
\begin{align} \label{eq:kerpix}
\psi(\ker(\pi_x))\subseteq \ker(\pi_x) .
\end{align} 
The point is that this inclusion holds for general $x\in X$, whereas $\psi$ does not depend on $x$.
Since $X\subseteq A$ is an ample divisor, this condition forces $\psi$ to be a multiple of the identity as follows.

The kernel of $H^{1,0}(A)\longrightarrow \Omega^1_{X,x}$ is generated by a $1$-form
\[
\alpha(x):=\sum_{i=1}^g \lambda_i(x)\cdot dz_i ,
\]
where $z_1,\ldots ,z_g$ denote coordinates on $\C^g$ with $A=\C^g\slash \Gamma$.
Hence,
\[
\ker(\pi_x)= \alpha(x)\wedge H^{k-1,0}(A) .
\]
Since $X$ is an ample divisor, the Gauss map
\[
G_X:X\dashedlongrightarrow \mathbb P^{g-1} ,\ \ x\mapsto [\alpha(x)]
\]
is dominant, where we identify $\Omega_{A,x}^{1}$ via translation with $\Omega_{A,0}^{1}$. 
Since (\ref{eq:kerpix}) holds for general $x\in X$, we conclude that for general (hence for all) $\alpha\in H^{1,0}(A)$, 
\begin{align} \label{eq:psi(alpha.beta)}
\psi(\alpha\wedge H^{k-1,0}(A))\subseteq \alpha \wedge H^{k-1,0}(A) .
\end{align}

Let $I=(i_1,\ldots ,i_k)$ be a $k$-tuple of integers $1\leq i_j\leq g$ with $i_j\neq i_l$ for $j\neq l$ and consider the corresponding $k$-form $dz_I:=dz_{i_1}\wedge\ldots \wedge dz_{i_k}$.
Applying (\ref{eq:psi(alpha.beta)}) to $\alpha=dz_{i_j}$, we conclude
\begin{align} \label{eq:lambda(I)}
\psi(dz_I)=\lambda(I)\cdot dz_I
\end{align}
for some constant $\lambda(I)\in \C$.
Let $n$ be an integer with $1\leq n\leq g$ and $n\neq i_j $ for all $j$.
Then, 
\[
\psi((dz_n+dz_{i_1})\wedge dz_{i_2}\wedge \ldots \wedge dz_{i_k}) 
\]
is by (\ref{eq:psi(alpha.beta)}) a multiple of $dz_n+dz_{i_1}$ and so it follows from (\ref{eq:lambda(I)}) that 
$$
\lambda(I)=\lambda(n,i_2,\ldots ,i_k).
$$
Since, by definition, $\lambda(I)$ is invariant under permutation of the indices $i_j$, it follows that $\lambda(I)=\lambda$ does not depend on $I$.
That is, 
\begin{align*} 
\psi=\lambda\cdot \id .
\end{align*}
By Lemmas \ref{lem:ast} and \ref{lem:nonzero}, $\lambda \neq 0$, which finishes the proof of Proposition \ref{prop:psi}.
\end{proof}

\subsection{Proof of Theorem \ref{thm:prop:P}} 
Fix $1\leq k \leq d$, and let $T^k(A)\subseteq H^k(A,\Q)$ be the transcendental lattice, i.e.\ the smallest rational sub-Hodge structure with $T^{k,0}(A)=H^{k,0}(A)$; the transcendental lattice $T^{k}(\widetilde X)\subseteq H^k(\widetilde X,\Q)$ is defined similarly.
By Lemma \ref{lem:Hk0X}, $j^{\ast}:H^k(A,\Q) \longrightarrow H^k(\widetilde X,\Q)$ induces an isomorphism  on the transcendental lattices
$$
j^{\ast}:T^k(A)\stackrel{\sim}\longrightarrow T^{k}(\widetilde X),
$$
and so we can define its inverse $(j^{\ast})^{-1}$.
In particular, the homomorphism 
$$
\psi=(j^\ast)^{-1}\circ \tilde \mu_\ast\circ q^\ast :H^{k,0}(A)\longrightarrow H^{k,0}(A)
$$ 
from Proposition \ref{prop:psi} extends to an endomorphism of the rational Hodge structure $T^k(A)$, and so $\psi=\lambda_k\cdot \id$ for some $\lambda_k\in \Q^{\times}$.
Hence, by Lemma \ref{lem:psi}, 
$$
c^t_k(V,W)(x)=\lambda_k \cdot \id \ \ \text{and}\ \ c_k(V,W)(x)=\lambda_k \cdot \id ,
$$  
where we used that $c_k(V,W)(x)$ is the transpose of $c^t_k(V,W)(x)$. 
By construction, $\lambda_k\in \Q^{\times}$ does not depend on $x$, which finishes the proof of Theorem \ref{thm:prop:P}.

\section{Proof of Theorem \ref{thm:P}} 
In this section we prove Theorem \ref{thm:P}; our notation will always be that of Section \ref{sec:P}. 

Since $q^{\ast}=\widetilde r^{\ast}\circ \pr_1^{\ast} \circ i^{\ast}$, Proposition \ref{prop:psi} implies that $i^{\ast}:H^{d,0}(A)\longrightarrow H^{d,0}(\widetilde V)$ is injective and so $V$ is nondegenerate by \cite[Lemma II.1]{ran}.
This proves item (\ref{item:P:1}) of Theorem \ref{thm:P} by symmetry in $V$ and $W$.
Moreover, item (\ref{item:P:2}) of Theorem \ref{thm:P} follows directly from Theorem \ref{thm:prop:P}, because $c_d(V,W)$ is a sum of $s=\deg(f)$ many rank one projectors.

It remains to prove items (\ref{item:P:3:0})--(\ref{item:P:4}) of Theorem \ref{thm:P}.
For this we assume from now on that $\deg(f)=\binom{g-1}{d}$ is minimal.
For technical reasons, we start with (\ref{item:P:3}) and (\ref{item:P:4}).

\subsection{Proof of item (\ref{item:P:3})}
Our goal is to prove that
\[
i^\ast: H^{d,0}(A)\stackrel{\sim}\longrightarrow H^{d,0}(\widetilde V) 
\]
is an isomorphism, which clearly implies $p_g(V)=\binom{g}{d}$.

By Proposition \ref{prop:psi}, $(j^{\ast})^{-1}\circ \tilde \mu_\ast \circ q^\ast:  H^{d,0}(A)\longrightarrow H^{d,0}(A)$ is an isomorphism.
Since $q^{\ast}=\widetilde r^{\ast}\circ \pr_1^{\ast} \circ i^{\ast}$, this implies (as we have already noted above) that $i^{\ast}$ is injective and
\begin{align} \label{eq:cor:P}
(j^{\ast})^{-1}\circ\tilde \mu_\ast\circ \tilde r^\ast \circ \pr_1^\ast :H^{d,0}(\widetilde V)\longrightarrow H^{d,0}(A) 
\end{align} 
is an isomorphism on the subspace $i^\ast H^{d,0}(A)$. 
In order to prove that $i^\ast$ is an isomorphism on $(d,0)$-classes, it therefore suffices to see that the morphism in (\ref{eq:cor:P}) is injective.
To this end, let us consider an element $\omega\in H^{d,0}(\widetilde V)$ in the kernel of (\ref{eq:cor:P}) and think about $\omega$ as holomorphic $d$-form on $\widetilde V$.
Applying $r_\ast\circ j^{\ast}$ shows
$$
r_{\ast}  \circ  \tilde \mu_\ast\circ \tilde r^\ast \circ \pr_1^\ast(\omega) =0 .
$$
By (\ref{eq:mu}), $r_{\ast}  \circ  \tilde \mu_\ast\circ \tilde r^\ast = \mu_\ast$, which implies
$$
\mu_\ast\circ \pr_1^\ast(\omega)=0.
$$ 

Let $x\in X$ be a general point with $f^{-1}(x)=\left\{(v_1,w_1),\ldots ,(v_s,w_s)\right\}$.
By Theorem \ref{thm:prop:P}, $c_d(V,W)(x)=\sum_{i=1}^s \Lambda^d P_i$ is an isomorphism.
The image of $\Lambda^dP_i$ is the line $\Lambda^d T_{V,v_i}$ in $\Lambda^d T_{X,x}$.
By assumptions, $s=\dim(\Lambda^d T_{X,x})$, and so it follows that the lines $\Lambda^d T_{V,v_i}\subseteq \Lambda^dT_{X,x}$ are linearly independent for $i=1,\ldots ,s$.
Therefore, $\mu_\ast\circ \pr_1^\ast(\omega)=0$ implies $\omega_{v_i}=0$ for each $i$ by the definition of the trace map $\mu_{\ast}$, see Section \ref{subsec:trace}.
Since $x\in X$ is general, $\omega_{v}=0$ for general $v\in V$ and so $\omega=0$.
This proves that
\[
i^\ast: H^{d,0}(A)\stackrel{\sim}\longrightarrow H^{d,0}(\widetilde V) 
\]
is an isomorphism, as we want.

\subsection{Proof of item (\ref{item:P:4})}
We need to see that $V$ has property $(\mathcal P)$ with respect to $W$.
This follows from Theorem \ref{thm:prop:P} and the nondegeneracy of $V$ and $W$ (item (\ref{item:P:1}) of Theorem \ref{thm:P}, proven above) by a similar argument as in the proof of Theorem 3.1 in \cite{debarre}.
For convenience of the reader, we give the details in the following.

By Theorem \ref{thm:prop:P},
\[
c_d(V,W)(x)=\sum_{i=1}^s \Lambda^dP_i=\lambda_d\cdot \id :\Lambda^d T_{X,x}\longrightarrow \Lambda^d T_{X,x}.
\]
Here, $\lambda_d\in \Q^{\times}$ and $s=\deg(f)=\binom{g-1}{d}$ by assumptions.
Moreover, $\Lambda^dP_i: \Lambda^d T_{X,x}\longrightarrow \Lambda ^d T_{X,x}$ is the homomorphism whose image is the line $\Lambda^dT_{V,v_i}$ and whose kernel is 
\[
\ker(\Lambda ^dP_i)=T_{W,w_i}\wedge \Lambda^{d-1}T_{X,x} .
\]
Therefore, for suitable generators $\alpha_i\in \Lambda^dT_{V,v_i}$ and $\beta_i\in \Lambda^{g-1-d} T_{W,w_i}$, we have
\[
\Lambda^dP_i(\alpha)=(  \alpha \wedge \beta_i) \cdot \alpha_i \ \ \text{for all}\ \ \alpha\in \Lambda^dT_{X,x} ,
\]
where we use a fixed isomorphism $\Lambda^{g-1}T_{X,x}\cong \C$ to identify $ \alpha \wedge \beta_i$ with a scalar.
Since $\sum_i\Lambda^d P_i=\lambda_d\cdot \id$, the above formula yields 
\[
\lambda_d\cdot \alpha_j=\sum_{i=1}^s  (\alpha_j\wedge \beta_i) \cdot \alpha_i ,
\]
for all $j=1,\ldots ,s$.
This shows 
$
\alpha_j\wedge \beta_i =0
$
for all $i\neq j$, because  $\alpha_1,\dots ,\alpha_s$ is a basis of $\Lambda^d T_{X,x}$.
That is,
\begin{align} \label{eq:TWwedgeTV}
\Lambda^{g-1-d} T_{W,w_j}\wedge \Lambda^d T_{V,v_i}=0
\end{align}
for all $i\neq j$.

In the above equality, we can put $i=1$ and $j=2,\ldots ,s$.
Fixing $v_1$ and moving $w_1$ in a sufficiently small analytic neighborhood, the points $(v_i,w_i)$ with $v_i+w_i=v_1+w_1$ move in a unique way.
Since $W$ is nondegenerate by item (\ref{item:P:1}) above, its Pl\"ucker image is via the Gauss map not contained in any hyperplane. 
Identity (\ref{eq:TWwedgeTV}) shows therefore that the locus that is described by $w_i$ with $i\geq 2$ when $w_1$ moves and $v_1$ is fixed is of dimension less than $\dim(W)$.
This proves that $V$ has property $(\mathcal P)$ with respect to $W$, which finishes the proof of Theorem \ref{thm:P}.

\subsection{Proof of item (\ref{item:P:3:0})}
By Lemma \ref{lem:Hk0X}, in order to prove $h^0(A,\mathcal O_A(X))=1$, we need to see that 
\begin{align} \label{eq:item:P:3:0}
j^{\ast}:H^{g-1,0}(A)\longrightarrow H^{g-1,0}(\widetilde X)
\end{align}
is surjective. 
Recall from (\ref{diag:mu}) that the addition morphism $f$ induces a surjective morphism $\tilde \mu: \widetilde{ V \times W}\longrightarrow \widetilde X$ between smooth models of $V\times W$ and $X$.
Since
$$
\tilde \mu_{\ast}\circ \tilde \mu^{\ast} : H^{g-1,0}(\widetilde X)\longrightarrow H^{g-1,0}(\widetilde X)
$$
is an isomorphism (it is multiplication by $\deg(\tilde \mu)$), the composition
\begin{align} \label{eq:muomu}
H^{d,0}(\widetilde V)\otimes H^{g-1-d,0}(\widetilde W) \cong H^{g-1,0}(\widetilde{V\times W})\stackrel{\tilde \mu_{\ast} }\longrightarrow H^{g-1,0}(\widetilde X)
\end{align}
is surjective.
We have seen in the proof of item (\ref{item:P:3}) that the natural pullback maps 
$$
H^{d,0}(A)\longrightarrow H^{d,0}(\widetilde V) \ \ \text{and} \ \ H^{g-d-1,0}(A)\longrightarrow H^{g-d-1,0}(\widetilde W) 
$$ 
are isomorphisms.
The following lemma proves therefore that (\ref{eq:item:P:3:0}) and (\ref{eq:muomu}) have the same images.
Since (\ref{eq:muomu}) is surjective, so is (\ref{eq:item:P:3:0}), which concludes the proof of item (\ref{item:P:3:0}).

\begin{lemma}
Let $\alpha\in H^{d,0}(A)$ and $\beta\in H^{g-d-1,0}(A)$ and consider the corresponding holomorphic $(g-1)$-form $\alpha |_V\otimes \beta |_{W}$ on the smooth part of $V\times W$.
Then, the image
$$
f_{\ast}(\alpha |_V\otimes \beta |_{W})\in H^0(X^{\sm},\Omega^{g-1}_{X^{\sm}})
$$
via the trace map $f_{\ast}$ is a nonzero multiple of the restriction of $\alpha\wedge \beta\in H^{g-1,0}(A)$ to $X^{\sm}$.  
\end{lemma}

\begin{proof}
Let $x\in X^{\sm}$ be a general point and write 
$
f^{-1}(x)=\left\{(v_1,w_1),\dots ,(v_s,w_s)\right\} 
$. 
Then,
$$
f_{\ast}(\alpha |_V\otimes \beta |_{W})_x=\sum_{i=1}^s\alpha |_{V,v_i}\wedge \beta |_{W,w_i} ,
$$
where by slight abuse of notation,  
$\alpha |_{V,v_i}\in \Omega_{X,x}^d$
denotes the composition
$$
\Lambda ^d T_{X,x}\longrightarrow \Lambda ^d T_{V,v_i}\longrightarrow \C ,
$$
where the first arrow is the projection induced by the direct sum decomposition $T_{X,x}=T_{V,v_i}\oplus T_{W,w_i}$, and the second arrow is given by the restriction of the $d$-form $\alpha$ to $V$; $\beta |_{W,w_i}\in \Omega_{X,x}^{g-d-1}$ is defined similarly. 

For $i\neq j$, the intersection $T_{W,w_i}\cap T_{V,v_j}$ is nonzero by (\ref{eq:TWwedgeTV}), and so we can pick a nonzero vector $\nu_{ij} \in T_{X,x}$ which lies in that intersection.
Completing $\nu_{ij}$ to a basis of $T_{X,x}$ shows then that $\alpha |_{V,v_i}\wedge \beta |_{W,w_j}=0$ for all $i\neq j$, because the contractions of $\alpha |_{V,v_i} $ and $\beta |_{W,w_j} $ with $\nu_{ij}$ both vanish.
This implies
$$
f_{\ast}(\alpha |_V\otimes \beta |_{W})_x=\left( \sum_{i=1}^s\alpha |_{V,v_i} \right)\wedge \left(\sum_{j=1}^s \beta |_{W,w_j}\right) . 
$$ 
By Theorem \ref{thm:prop:P} there are nonzero constants $\lambda,\lambda'\in \Q^{\times}$, not depending on $x$, such that
$$
\sum_{i=1}^s\alpha |_{V,v_i}=\lambda\cdot \alpha |_{X,x}\ \ \text{and}\ \ \sum_{j=1}^s \beta |_{W,w_j}=\lambda'\cdot \beta |_{X,x}.
$$ 
Hence, $f_{\ast}(\alpha |_V\otimes \beta |_{W})$ coincides with the restriction of $\lambda\lambda'\cdot (\alpha \wedge \beta)$ to $X^{\sm}$, as we want.
\end{proof}

\section{Applications}

In this section we draw some consequences from Theorem \ref{thm:P}.
In particular, we give a proof of Theorem \ref{thm:GV}, stated in the introduction.

\begin{corollary} \label{cor:singular}
Let $A$ be a $g$-dimensional abelian variety, and let $V,W\subseteq A$ be closed geometrically nondegenerate subvarieties of dimensions $d$ and $g-1-d$, respectively. 
Suppose that one of the following holds:
\begin{enumerate}
\item $V$ or $W$ is degenerate;
\item the addition morphism $f:V\times W\longrightarrow V+W$ has degree $\deg(f)<\binom{g-1}{d}$;
\item $\deg(f) = \binom{g-1}{d}$ and $V+W$ is not a theta divisor on $A$ (e.g.\ $A\ncong \hat A$).
\end{enumerate}
Then  $X=V+W$ is an ample divisor with non-rational singularities. 
\end{corollary}
\begin{proof}
This follows from Theorem \ref{thm:P} and Corollary \ref{cor:V+W=ample}.
\end{proof}

The following two applications of Theorem \ref{thm:P} turn out to imply Theorem \ref{thm:GV}. 

\begin{corollary} \label{cor:P}
Let $(A,\Theta)$ be a $g$-dimensional indecomposable ppav, and let $V,W\subset A$ be subvarieties of minimal classes $[V]=\frac{\theta^{g-d}}{(g-d)!}$ and $[W]=\frac{\theta^{d+1}}{(d+1)!}$.
If $V+W=\Theta$, then 
\begin{enumerate}
\item $p_g(V)=\binom{g}{d}$ and $p_g(W)=\binom{g}{d+1}$;
\item $V$ has property $(\mathcal P)$ with respect to $W$ and viceversa.
\end{enumerate}
\end{corollary} 
\begin{proof}
Let $f:V\times W\longrightarrow \Theta$ be the addition map and recall that $\Theta$ has only rational singularities \cite{ein-laz}.
By the definition of the Pontryagin product (see Section \ref{subsec:star}),
$$
\binom{g-1}{d}\cdot \theta=[V]\star[W]=\deg(f)\cdot [V+W]=\deg(f)\cdot \theta .
$$
Hence, $\deg(f)=\binom{g-1}{d}$ and so the result follows from Theorem \ref{thm:P}.
\end{proof}

\begin{proposition} \label{prop:minclass}
Let $(A,\Theta)$ be an indecomposable ppav of dimension $g$, and let $V,W\subseteq A$ be closed pure-dimensional subschemes  of minimal classes $[V]=\frac{\theta^{g-d}}{(g-d)!}$ and $[W]=\frac{\theta^{d+1}}{(d+1)!}$.
If for each closed point $(v,w)\in V\times W$, $v+w\in \Theta$, then $V$ and $W$ are irreducible and generically reduced.
\end{proposition}

\begin{proof} 
Note first that the statement is true if $d=0$ or $d=g-1$.
Indeed, it suffices by symmetry to deal with $d=0$.
Since $[V]=1$, $V$ is a single reduced point $v$ and so $W^{\red}=\Theta-v$ is a translate of $\Theta$, hence irreducible. 
Since $[W]=\theta$, $W$ is generically reduced and so the only nontrivial statement is $p_g(\Theta)=g$, which is an easy consequence of the fact that $\Theta$ has only rational singularities \cite{ein-laz}.
In what follows, we may thus assume that $V$ and $W$ are positive-dimensional.

Let $V_i\subseteq V^{\red}$ and $W_j\subseteq W^{\red}$ be the irreducible components of the reduced schemes $V^{\red}$ and $W^{\red}$.
Then the cohomology classes of $V$ and $W$ are given by 
$
[V]=\sum_i a_i[V_i]
$
and 
$
[W]=\sum_jb_j[W_j]
$
for some positive integers $a_i$ and $b_j$.

By assumptions $V_i+W_j\subseteq \Theta$ for all $i$ and $j$. 
Since $V$ and $W$ have minimal classes,
\begin{align} \label{eq:Vi+Wi}
\binom{g-1}{d}\cdot \theta =[V]\star [W] = \sum_{i,j}a_ib_j[V_i]\star[W_j]= \sum_{i,j}a_ib_j c_{ij} \cdot \theta,
\end{align}
where 
$$
c_{ij}=
\begin{cases}
\deg(V_i\times W_j \longrightarrow \Theta)\ \ &\text{if }\  V_i+W_j=\Theta ,\\
0 \ \ &\text{else},
\end{cases}
$$
see Section \ref{subsec:star} above.
It follows that there is at least one pair of indices $(i_0,j_0)$ such that $c_{i_0j_0}\neq 0$
and so $V_{i_0}+W_{j_0}=\Theta$.

By \cite{ein-laz}, $\Theta$ has only rational singularities.
Item (\ref{item:P:1}) in Theorem \ref{thm:P} implies therefore that $V_{i_0}$ and $W_{j_0}$ are both nondegenerate, hence geometrically nondegenerate.
It follows from Lemma \ref{lem:geomnondeg} hat $V_{i_0}+W_j=\Theta$ for all $j$ and similarly $V_i+W_{j_0}=\Theta$ for all $i$.
Hence, for all $i$ and $j$, $V_i$ and $W_j$ are nodegenerate by item (\ref{item:P:1}) in Theorem \ref{thm:P} and so $V_i+W_j=\Theta$ for all $i$ and $j$. 
Therefore, item (\ref{item:P:2}) in Theorem \ref{thm:P} implies $c_{ij}\geq \binom{g-1}{d}$ for all $i$ and $j$.
It then follows from (\ref{eq:Vi+Wi}) that $V$ and $W$ are irreducible and generically reduced.
\end{proof}

\begin{proof}[Proof of Theorem \ref{thm:GV}]
Let $(A,\Theta)$ be an indecomposable $g$-dimensional ppav, and let $Z\subset A$ be a $d$-dimensional geometrically nondegenerate closed GV-subscheme with theta dual $V(Z)$.
Pareschi and Popa proved that $Z$ and $V(Z)$ are pure-dimensional Cohen--Macauly subschemes of $A$ of minimal cohomology classes $[Z]=\frac{\theta^{g-d}}{(g-d)!}$ and $[V(Z)]=\frac{\theta^{d+1}}{(d+1)!}$, see Theorem \ref{thm:pareschi-popa} above.
Moreover, for all closed points $(x,y)\in Z\times V(Z)$, $x-y\in\Theta$, see \cite[p.\ 216]{pareschi-popa}. 
Theorem \ref{thm:GV} follows therefore from Proposition \ref{prop:minclass} and Corollary \ref{cor:P}, because Cohen--Macauly schemes have no embedded points, see \cite[Theorem 17.3.(i)]{matsumura}. 
\end{proof}


\begin{remark} \label{rem:Fano} 
Let $F\subset (JY,\Theta_Y)$ be the Abel--Jacobi embedded Fano surface of lines on a smooth cubic threefold $Y$. 
Clemens and Griffiths \cite{clemens-griffiths} proved the the subtraction morphism $f:F\times F\longrightarrow \Theta_Y=F-F$ has degree $6$.
It therefore follows from Theorem \ref{thm:P} that $F$ has property $(\mathcal P)$ with respect to $-F$. 
This is a nontrivial fact which can also be deduced from the description of the fibers of the difference map $F\times F\longrightarrow \Theta_Y$ in terms of double six configurations of lines on a smooth cubic surface, see \cite[p.\ 348]{clemens-griffiths}.
\end{remark}



\section*{Acknowledgment}
I am grateful to both referees whose suggestions significantly improved the exposition.
I also thank Mart\'i Lahoz and Mihnea Popa for conversations and comments on a preliminary version of this paper.  
The author is member of the SFB/TR 45.


\begin{thebibliography}{9}   

%



\bibitem{CMPS}
S.\ Casalaina-Martin, M.\ Popa and S.\ Schreieder, {\em Generic vanishing and minimal cohomology classes on abelian fivefolds}, Preprint 2016, arXiv:1602.06231.

\bibitem{clemens-griffiths}
C.\ H.\ Clemens and P.\ A.\ Griffiths, {\em The intermediate Jacobian of the cubic threefold}, Ann.\ Math.\ \textbf{95} (1972), 281--356.


\bibitem{debarre}
O.\ Debarre, {\em Minimal cohomology classes and Jacobians}, J.\ Alg.\ Geom.\ \textbf{4} (1995), no. 2, 321--335.

\bibitem{debarre-tores}
O.\ Debarre, {\em Tores et vari \'et \'es ab \'eliennes complex}, Cours Sp\'ecialis\'es 6, Soci\'et\'e Math\'ematique de France, EDP Sciences, 1999.
 

\bibitem{debarre-ein-laz-voisin}
O.\ Debarre, L.\ Ein, R.\ Lazarsfeld and C.\ Voisin, {\em Pseudoeffective and nef classes on abelian varieties}, Compos.\ Math.\ \textbf{147} (2011), 1793--1818. 


\bibitem{ein-laz}
L.\ Ein and R.\ Lazarsfeld, {\em Singularities of theta divisors and the birational geometry of irregular varieties}, J.\ Amer.\ Math.\ Soc.\ \textbf{10} (1997), 243--258. 

\bibitem{eisenbud-harris}
D.\ Eisenbud and J.\ Harris, {\em On Varieties of minimal degree (a centennial account)}, Proc.\ Symp.\ Pure Math.\ \textbf{46} Am.\ Math.\ Soc.\ (1987), 3--13.



\bibitem{gulbrandsen-lahoz}
M.\ G.\ Gulbrandsen and M.\ Lahoz, {\em Finite subschemes of abelian varieties and the Schottky problem}. Ann.\ Inst.\ Fourier (Grenoble) \textbf{61} (2011), 2039--2064.


\bibitem{griffiths}
P.\ A.\ Griffiths, {\em Variations on a theorem of Abel}, Invent.\ Math.\ \textbf{35} (1976), 321--390.


\bibitem{hoering}
A.\ H\"oring, {\em M-regularity of the Fano surface}, C.\ R.\ Math.\ Acad.\ Sci.\ Paris \textbf{344} (2007), 691--696. 

 

\bibitem{lombardi-tirabassi}
L.\ Lombardi and S.\ Tirabassi, {\em GV-subschemes and their embeddings in principally polarized abelian varieties}, Math.\ Nachr.\ \textbf{288} (2015), 1405--1412.



\bibitem{matsumura}
H.\ Matsumura, {\em Commutative ring theory}, Cambridge University Press, Cambridge, 1986.


\bibitem{pareschi-popa-JAMS}
G.\ Pareschi and M.\ Popa, {\em Regularity on abelian varieties I}, J.\ Am.\ Math.\ Soc.\ \textbf{16} (2003), 285--302.


\bibitem{pareschi-popa3}
G.\ Pareschi and M.\ Popa, {\em Castelnuovo theory and the geometric Schottky problem}, J.\ Reine Angew.\ Math.\ \textbf{615} (2008), 25--44.


\bibitem{pareschi-popa}
G.\ Pareschi and M.\ Popa, {\em Generic vanishing and minimal cohomology classes on abelian varieties}, Math.\ Ann.\ \textbf{340} (2008), 209--222.


\bibitem{pareschi-popa-AJM}
G.\ Pareschi and M.\ Popa, {\em GV-sheaves, Fourier--Mukai transform, and generic vanishing}, Amer.\ J.\ Math.\ \textbf{133} (2011), 235--271. 

\bibitem{ran}
Z.\ Ran, {\em On subvarieties of abelian varieties}, Invent.\ Math.\ \textbf{62} (1981), 459--479.

\bibitem{ran2}
Z.\ Ran, {\em A characterization of five-dimensional Jacobian varieties}, Invent.\ Math.\ \textbf{67} (1982), 395--422.
 

\bibitem{schreieder}
S.\ Schreieder, {\em Theta divisors with curve summands and the Schottky problem}, Math.\ Ann.\ 2016; doi:10.1007/s00208-015-1287-8. 



\bibitem{voisin1}
C. Voisin, {\em Hodge Theory and Complex Algebraic Geometry, I}, Cambridge University Press, Cambridge, 2002. 



\end{thebibliography}
\end{document}